\documentclass{amsart}

\RequirePackage{fix-cm}
\usepackage{graphicx}
\usepackage{amsmath, amsopn,amstext,amscd,amsfonts,amssymb}
\usepackage{dsfont}
\usepackage{comment}
\usepackage[active]{srcltx}
\usepackage{graphicx, epsfig, subfig}

\usepackage{textcomp}

\usepackage{algorithm}

\makeatletter
\def\BState{\State\hskip-\ALG@thistlm}
\makeatother

\def\downbar#1{
\setbox10=\hbox{$#1$}
            \dimen10=\ht10 \advance\dimen10 by 2.5pt
            \ifdim \dimen10<15pt 
               \advance\dimen10 by -0.5pt
               \dimen11=\dimen10
               \advance\dimen10 by 2.5pt
               \lower \dimen11
            \else \lower \ht10 \fi
            \hbox {\hskip 1.5pt \vrule height \dimen10 depth \dp10}}
\def\upbar#1{
\setbox10=\hbox{$#1$}
            \dimen10=\ht10 \advance\dimen10 by \dp10 \advance\dimen10 by 2.5pt
            \ifdim \dimen10<15pt 
                \advance\dimen10 by 2pt \fi
            \raise 2.5pt \hbox {\hskip -1.5pt \vrule height \dimen10}}

\usepackage{multicol}
\usepackage{colortbl}
\usepackage{exscale}
\usepackage{amssymb,latexsym,amsthm,amsfonts,color,fancyhdr,mathrsfs}
\usepackage{lscape}
\usepackage{rotating}
\usepackage{caption}

\newtheorem{definition}{\bf Definition}[section]
\newtheorem{theorem}{\bf Theorem}[section]
\newtheorem{proposition}{\bf Proposition}[section]

\newtheorem{corollary}{\bf Corollary}[section]
\newtheorem{remark}{\bf Remark}[section]

\numberwithin{equation}{section}
\begin{document}
\title[Classical orthogonal polynomials]{On First type characterizations of Askey-Wilson polynomials}

\author{D. Mbouna}
\address{D. Mbouna \\Department of Mathematics, Faculty of Sciences, University of Porto, Campo Alegre st., 687, 4169-007 Porto, Portugal}
\email{dieudonne.mbouna@fc.up.pt}

%
\author{A. Suzuki}
\address{A. Suzuki \\University of Coimbra, CMUC, Dep. Mathematics, 3001-501 Coimbra, Portugal}
\email{asuzuki@uc.pt}

\date{\today}

\maketitle
In this chapter we characterize Askey-Wilson polynomials including specific and limiting cases of them by some structure relations of the first type.

%
%

\section{Introduction}\label{introduction}
Classical orthogonal polynomial sequences (OPS) with respect to the Askey-Wilson operator are known as Askey-Wilson polynomials including special or limiting cases of them. These orthogonal polynomials have been subject of many interests along the last decade because of their importance and application in other mathematical branches and physical sciences. Among their most useful properties, we mention the three term recurrence relation, their generating functions, their moments, asymptotic properties of their zeros, their {\it derivatives} sequences, the second order linear difference equation and the Rodrigues-type formula. Such OPS are also described by characterization theorems. These are structure relations combining linear and non-linear consecutive terms of the OPS including their derivatives. For more reference about such theorems, we refer the reader to the following books \cite{A1990, KLS2010, L2005}. As a matter of fact, we mention that the Al-Salam polynomials were discovered via characterization theorems (see \cite{AC1976}). In this chapter, we investigate on those of the first type. That is, those involving only the first derivative of the polynomial. For example, in 1972, Al-Salam and Chihara proved (see \cite{Al-Salam-1972}) that $(P_n)_{n\geq 0}$ is a $\mathrm{D}$-classical orthogonal polynomial sequence, namely Hermite, Laguerre, Bessel or Jacobi families, if and only if 
\begin{align}
(az^2+bz+c)\mathrm{D}P_n(z)=(a_nz+b_n)P_n(z)+c_nP_{n-1}(z)\quad (c_n\neq 0)\;, \label{very-classical}
\end{align}
where $\mathrm{D}=d/dz$. Now, replace $\mathrm{D}$ in \eqref{very-classical} by the following Askey-Wilson operator

\begin{align*}
\mathcal{D}_q\,f(x)=\frac{\breve{f}\big(q^{1/2} e^{i\theta}\big)
-\breve{f}\big(q^{-1/2} e^{i\theta}\big)}{\breve{e}\big(q^{1/2}e^{i\theta}\big)-\breve{e}\big(q^{-1/2} e^{i\theta}\big)},
\end{align*}
where, for each polynomial $f$,  $\breve{f}(e^{i\theta})=f(\cos \theta)$ and $e(x)=x$ (see \cite[Section 12.1]{I2005}). The problem of characterizing such OPS was posed by Ismail (see \cite[Conjecture 24.7.8]{I2005}). The case $a=b=0$ and $c=1$ was considered by Al-Salam (see \cite{A-1995}). This problem was addressed in its full generality in \cite{KDPconj}, which leads to a characterization of continuous $q$-Jacobi, Chebyshev of the first kind and some special cases of the Al-Salam-Chihara polynomials. Our motivation here is to obtain a full characterization of Askey-Wilson polynomials similar to \eqref{very-classical}. We consider the averaging operator defined by
\begin{align*}
\mathcal{S}_q f(x)=\frac{\breve{f}\big(q^{1/2} e^{i\theta}\big)
+\breve{f}\big(q^{-1/2} e^{i\theta}\big)}{2}.
\end{align*}
We adopt the following notations. Take $0<q<1$, define $z=x(s)=\cos \theta=(q^s+q^{-s})/2$, with $s\in \mathbb{C}$.
For our purpose, instead of \eqref{very-classical}, we first consider the following difference equation
\begin{align}\label{open-problem}
&(az^2+bz+c)\mathcal{D}_q P_n(z)=a_n\mathcal{S}_q P_{n+1}(z)+b_n\mathcal{S}_q P_n(z)+c_n\mathcal{S}_q P_{n-1}(z)\;,
\end{align}
with $c_n\neq 0$. Our objective is then to characterize all OPS that satisfy \eqref{open-problem}. 
We will also consider the following case
\begin{align*}
&(az^2+bz+c)\mathcal{D}_q \mathcal{S}_q P_n(z)=a_n P_{n+1}(z)+b_n P_n(z)+c_n P_{n-1}(z)\;.
\end{align*}
We recall that the monic Askey-Wilson polynomial, $(Q_n(\cdot; a_1, a_2, a_3, a_4 | q))_{n\geq 0}$, satisfy \eqref{TTRR_relation} (see \cite[(14.1.5)]{KLS2010}) with
\begin{align*}
2B_n &= a_1+\frac{1}{a_1}-\frac{(1-a_1a_2q^n)(1-a_1a_3q^n)(1-a_1a_4q^n)(1-a_1a_2a_3a_4q^{n-1})}{a_1(1-a_1a_2a_3a_4q^{2n-1})(1-a_1a_2a_3a_4q^{2n})}\\[7pt]
&\quad-\frac{a_1(1-q^n)(1-a_2a_3q^{n-1})(1-a_2a_4q^{n-1})(1-a_3a_4q^{n-1})}{(1-a_1a_2a_3a_4q^{2n-1})(1-a_1a_2a_3a_4q^{2n-2})},\\[7pt]
C_{n+1}&=(1-q^{n+1})(1-a_1a_2a_3a_4q^{n-1}) \\[7pt]
&\quad\times \frac{(1-a_1a_2q^n)(1-a_1a_3q^n)(1-a_1a_4q^n)(1-a_2a_3q^n)(1-a_2a_4q^n)(1-a_3a_4q^n)}{4(1-a_1a_2a_3a_4q^{2n-1})(1-a_1a_2a_3a_4q^{2n})^2 (1-a_1a_2a_3a_4q^{2n+1})}
\end{align*}
and subject to the following restrictions (see \cite{KDP2021}):
$$
\begin{array}l
(1-a_1a_2a_3a_4q^n)(1-a_1a_2q^n)(1-a_1a_3q^n) \\[7pt]
\qquad\quad\times(1-a_1a_4q^n)(1-a_2a_3q^n)(1-a_2a_4q^n)(1-a_3a_4q^n) \neq 0.
\end{array}
$$
The Rogers $q$-Hermite polynomials satisfy the TTRR \eqref{TTRR_relation} with 
$$B_n=0\;,\quad C_{n+1}=(1-q^{n+1})/4\quad \quad (n=0,1,\ldots)\;.$$
The structure of the chapter is as follows. Section 2 presents some basic facts of the algebraic theory of OPS together with some useful results. Section \ref{main} contains our main results. In Section \ref{example} we present a finer result for some special cases.  

\section{Preliminaries}
The algebraic theory of orthogonal polynomials was introduced by P. Maroni (see \cite{M1991}). In what follows we recall some basic facts n the theory. Let $\mathcal{P}$ be the vector space of all polynomials with complex coefficients
and let $\mathcal{P}^*$ be its algebraic dual. A simple set in $\mathcal{P}$ is a sequence $(P_n)_{n\geq0}$ such that $\mathrm{deg}(P_n)=n$ for each $n$. 
\begin{definition}
A simple set $(P_n)_{n\geq0}$ is called an OPS with respect to ${\bf u}\in\mathcal{P}^*$ if 
$$
\langle{\bf u},P_nP_m\rangle=\kappa_n\delta_{n,m}\quad(m=0,1,\ldots;\;\kappa_n\in\mathbb{C}\setminus\{0\}),
$$
where $\langle{\bf u},f\rangle$ is the action of ${\bf u}$ on $f\in\mathcal{P}$. In this case, we say that ${\bf u}$ is  regular. 
\end{definition}
\begin{definition}
The left multiplication of a functional ${\bf u}$ by a polynomial $\phi$ is defined by
$$
\left\langle \phi {\bf u}, f  \right\rangle =\left\langle {\bf u},\phi f  \right\rangle \quad (f\in \mathcal{P}).
$$
\end{definition}
\begin{proposition}\label{Propo-an-def}
Let ${\bf u} \in \mathcal{P}^*$ be a functional.  Assume that $\mathcal{P}$ is endowed with an appropriate strict inductive limit topology. Then ${\bf u}$ can be written in the sense of the weak topology in $\mathcal{P}^*$ as 
\begin{align*}
{\bf u} = \sum_{n=0} ^{\infty} \left\langle {\bf u}, P_n \right\rangle {\bf a}_n\;,
\end{align*}
where $({\bf a}_n)_{\geq 0}$ is the dual basis associated to sequence of simple set $(P_n)_{n\geq 0}$. In particular, if $(P_n)_{n\geq0}$ is a monic OPS with respect to ${\bf u}\in\mathcal{P}^*$. Then the corresponding dual basis is explicitly given by 
\begin{align}\label{expression-an}
{\bf a}_n =\left\langle {\bf u} , P_n ^2 \right\rangle ^{-1} P_n{\bf u}.
\end{align}
\end{proposition}
The following theorem is a part of a result known in the literature as Favard's Theorem.
\begin{theorem} 
A monic OPS $(P_n)_{n\geq 0}$ is characterized by the following three-term recurrence relation (TTRR):
\begin{align}\label{TTRR_relation}
P_{-1} (z)=0, \quad P_{n+1} (z) =(z-B_n)P_n (z)-C_n P_{n-1} (z) \quad (C_n \neq 0),
\end{align}
and, therefore,
\begin{align}\label{TTRR_coefficients}
B_n = \frac{\left\langle {\bf u} , zP_n ^2 \right\rangle}{\left\langle {\bf u} , P_n ^2 \right\rangle},\quad C_{n+1}  = \frac{\left\langle {\bf u} , P_{n+1} ^2 \right\rangle}{\left\langle {\bf u} , P_n ^2 \right\rangle}.
\end{align}
\end{theorem}

Along this work, we will adopt the following definition (see \cite{FK-NM2011}).
\begin{definition}
The Askey-Wilson and the averaging operators  induce two elements on $\mathcal{P}^*$, say $\mathbf{D}_q$ and $\mathbf{S}_q$, via the following definition : 
\begin{align*}
\langle \mathbf{D}_q{\bf u},f\rangle=-\langle {\bf u},\mathcal{D}_q f\rangle,\quad \langle\mathbf{S}_q{\bf u},f\rangle=\langle {\bf u},\mathcal{S}_q f\rangle.
\end{align*}
\end{definition}
\begin{proposition}\cite{KDP2021, FK-NM2011, MFN2017} 
Let $f,g\in\mathcal{P}$ and ${\bf u}\in\mathcal{P}^*$. Then the following properties hold:
\begin{align}
\mathcal{D}_q \big(fg\big)&= \big(\mathcal{D}_q f\big)\big(\mathcal{S}_q g\big)+\big(\mathcal{S}_q f\big)\big(\mathcal{D}_q g\big), \label{def-Dx-fg} \\[7pt]
\mathcal{S}_q \big( fg\big)&=\big(\mathcal{D}_q f\big) \big(\mathcal{D}_q g\big)\texttt{U}_2  +\big(\mathcal{S}_q f\big) \big(\mathcal{S}_q g\big), \label{def-Sx-fg} \\[7pt]
\alpha \mathcal{S}_q ^2 f&=\mathcal{S}_q \big(\texttt{U}_1 \mathcal{D}_qf\big) +\texttt{U}_2\mathcal{D}_q ^2 f +\alpha f ,\label{def-Sx^2f}\\[7pt]
f\mathcal{D}_qg&=\mathcal{D}_q\left[ \Big(\mathcal{S}_qf-\frac{\texttt{U}_1}{\alpha}\mathcal{D}_qf \Big)g\right]-\frac{1}{\alpha}\mathcal{S}_q \Big(g\mathcal{D}_q f\Big) , \label{def-fDxg} \\[7pt]
\mathcal{D}_q ^n\mathcal{S}_qf &=\alpha_n \mathcal{S}_q\mathcal{D}_q ^n f+\gamma_n \texttt{U}_1\mathcal{D}_q ^{n+1}f   ,\label{def-DxnSxf}\\[7pt]
f{\bf D}_q {\bf u}&={\bf D}_q\left(\mathcal{S}_qf~{\bf u}  \right)-{\bf S}_q\left(\mathcal{D}_qf~{\bf u}  \right), \label{def-fD_x-u}\\[7pt]
\alpha \mathbf{D}_q ^n \mathbf{S}_q {\bf u}&= \alpha_{n+1} \mathbf{S}_q \mathbf{D}_q^n {\bf u}
+\gamma_n \texttt{U}_1\mathbf{D}_q^{n+1}{\bf u}, \label{DxnSx-u} 
\end{align}
with $n=0,1,\ldots$, where $\alpha =(q^{1/2}+q^{-1/2})/2$, $\texttt{U}_1 (z)=(\alpha ^2 -1)z$ and $\texttt{U}_2 (z)=(\alpha ^2 -1)(z^2-1)$. 
\end{proposition}

\begin{proposition}
The following relations hold. 
\begin{align}
\mathcal{D}_q z^n =\gamma_n z^{n-1}+u_nz^{n-3}+\cdots,\quad \mathcal{S}_q z^n =\alpha_n z^n+\widehat{u}_nz^{n-2}+\cdots, \label{Dx-xnSx-xn}
\end{align}
with $n=0,1,\ldots$, where $$\alpha_n= \mbox{$\frac12$}(q^{n/2} +q^{-n/2})\;,\quad \gamma_n = \frac{q^{n/2}-q^{-n/2}}{q^{1/2}-q^{-1/2}}$$ and, $u_n$ and $\widehat{u}_n$ are some complex numbers. 
\end{proposition}
\begin{proof}
By mathematical induction on $n\in \mathbb{N}$.
\end{proof}

Along this chapter we set $\gamma_{-1}:=-1$ and $\alpha_{-1}:=\alpha$.
We denote by $P_n ^{[k]}$ $(k=0,1,\ldots)$ the monic polynomial of degree $n$ defined by
\begin{align*}
P_n ^{[k]} (z)=\frac{\mathcal{D}_q ^k P_{n+k} (z)}{ \prod_{j=1} ^k \gamma_{n+j}} =\frac{\gamma_{n} !}{\gamma_{n+k} !} \mathcal{D}_q ^k P_{n+k} (z). 
\end{align*}
Here it is understood that $\mathrm{D}_x ^0 f=f $, empty product equals one, and $\gamma_0 !=1$, $\gamma_{n+1}!=\gamma_1\cdots \gamma_n \gamma_{n+1}$. 

\begin{proposition}
Let $({\bf a}^{[k]} _n)_{n\geq 0}$ be the dual basis associated to the sequence $(P_n ^{[k]})_{n\geq 0}$. Then
\begin{align}
{\bf D}_q ^k {\bf a}^{[k]} _n=(-1)^k \frac{\gamma_{n+k}!}{\gamma_n ! }{\bf a}_{n+k}\quad (k=0,1,\ldots). \label{basis-Dx-derivatives}
\end{align}
\end{proposition}

\begin{proof}
This result is obtained directly from Proposition \ref{Propo-an-def}. 
\end{proof}

We recall some useful theorems.
\begin{theorem}\cite[Theorem 20.1.3]{I2005}\label{T}\\
A second order operator equation of the form
\begin{align}\label{Ismail}
\phi(z)\mathcal{D}^2_q\, Y+\psi(z) \mathcal{S}_q \mathcal{D}_q \, Y+h(z)\, Y=\lambda_n\, Y
\end{align}
has a polynomial solution $Y_n(z)$ of exact degree $n$ for each $n=0,1,\dots$, if and only if $Y_n(z)$ is a multiple of the Askey-Wilson polynomials, or special or limiting cases of them. In all these cases $\phi$, $\psi$, $h$, and $\lambda_n$ reduce to
\begin{align*}
\phi(z)&=-q^{-1/2}(2(1+\sigma_4)z^2-(\sigma_1+\sigma_3)z-1+\sigma_2-\sigma_4),\\[7pt]
\psi(z)&=\frac{2}{1-q} (2(\sigma_4-1)z+\sigma_1-\sigma_3), \quad h(z)=0,\\[7pt]
\lambda_n&=\frac{4 q(1-q^{-n})(1-\sigma_4 q^{n-1})}{(1-q)^2}, 
\end{align*}
or a special or limiting case of it, $\sigma_j$ being the jth elementary symmetric function of the Askey-Wilson parameters. 
\end{theorem}

The following result is presented in \cite{MKKJ2019} as a partial answer to a conjecture posed by M. E. H. Ismail in \cite{I2005}. We also refer the reader to \cite{KCDM2023} for a short proof of this result.  
\begin{theorem}\label{Kenfack-Kerstin-result}
The only OPS, $(P_n)_{n\geq 0}$, for which 
\begin{align}\label{Ismail-conjecture}
\pi(z)\mathcal{D}_q ^2 P_n(z)=\sum_{j=n-2} ^{n+2} a_{n,j}P_j(z)\;,\quad\quad a_{n,n-2}\neq 0\;,\quad n=0,1,\ldots\;.
\end{align}
where $\pi$ is a polynomial of degree at most four, are the Askey-Wilson polynomials including special or limiting cases of them.
\end{theorem}

\section{Results}\label{main}
We are now in the position to prove our results. 
\begin{theorem} \label{propo-sol-q-quadratic}
If $(P_n)_{n\geq 0}$ is a monic OPS such that 
\begin{align}
(az^2+bz+c)\mathcal{D}_qP_{n}(z)= a_n\mathcal{S}_qP_{n+1}(z)+b_n\mathcal{S}_q P_{n}(z)+c_n\mathcal{S}_qP_{n-1}(z)\quad (n=0,1,\ldots)\;,\label{equation-case-deg-is-two}
\end{align}
with $c_n\neq 0$ for $n=0,1,\ldots$, where the constant parameters $a$, $b$ and $c$ are chosen such that
\small{
\begin{align}
&(4\alpha^2-1)aC_2C_3 
+\frac{r_3}{2}\Big[(B_0+B_1)^2 +4\alpha^2 (C_1-B_0B_1 +\alpha^2 -1) -2(2\alpha^2-1)C_2  \Big]=0\;,\label{condition-case-deg-is-two-1}
\end{align}
}
whenever $a\neq 0$, and
\begin{align}
aC_2C_3\Big(b_2+2aB_2+\frac{b}{\alpha}\Big) 
-r_3\left(a(B_2+B_1)C_2  +\frac{b}{\alpha}C_2  -\frac{r_2}{2}(B_1-B_0) \right)=0\;, \label{condition-case-deg-is-two-2}
\end{align}
$r_i=c_i+2aC_i$, $i=2,3$, then $(P_n)_{n\geq 0}$ are multiple of Askey-Wilson polynomials, or special or limiting cases of them. Moreover $(P_n)_{n\geq 0}$ satisfy \eqref{Ismail} with
\begin{align}
&\phi(z)=\mathfrak{a}z^2 +\mathfrak{b}z+\mathfrak{c}\;,~\psi(z)=z-B_0\;,~h(z)=0\;,~ \lambda_n=\gamma_n (\mathfrak{a}\gamma_{n-1}+\alpha_{n-1})\;,\label{expression-phi-psi-general-case}
\end{align}
where
\begin{align*}
&\mathfrak{a}=-\frac{aC_3 +(\alpha^2-1)r_3}{\alpha r_3}\;;\\
&\mathfrak{b}=-\frac{1}{2\alpha}\left(\Big(1-2a\frac{C_3}{r_3}\Big)\Big(B_0+B_1\Big)-2\alpha^2B_0  \right) \;;\\
&\mathfrak{c}=-\frac{1}{2\alpha}\left(\Big(1-2a\frac{C_3}{r_3}\Big)\Big(C_1-B_0B_1\Big)+C_1 +B_0 ^2\right)   \;;
\end{align*}
being $B_0$, $B_1$, $C_1$, $C_2$ and $C_3$ coefficients for the TTRR relation \eqref{TTRR_relation} satisfied by $(P_n)_{n\geq 0}$.
\end{theorem}  

\begin{proof}
Let $(P_n)_{n\geq 0}$ be a monic OPS with respect to the functional ${\bf u} \in \mathcal{P}^*$ and satisfying \eqref{equation-case-deg-is-two}. Under conditions \eqref{condition-case-deg-is-two-1}--\eqref{condition-case-deg-is-two-2} we obtain
$${\bf D}_q(\phi {\bf u})={\bf S}_q(\psi {\bf u})\;,$$
where $\phi$ and $\psi$ are given in \eqref{expression-phi-psi-general-case}. Since ${\bf u}$ is regular, this is equivalent to  
$$\phi(z) \mathcal{D}_q ^2 P_n(z) + \psi(z) \mathcal{S}_q\mathcal{D}_qP_n(z) =a_{n,n}P_n(z)~~\quad(n=1,2,\ldots)\;,$$ 
where
$a_{n,n}=\gamma_n(\mathfrak{a}\gamma_{n-1}+\alpha_{n-1})\neq 0$. The desired result follows by Theorem \ref{T}. A completed proof of this result with all details is available in \cite{DMAS2022}.
\end{proof}

For our second structure relation mentioned at the end of the introduction, the following result follows. We emphasize that this result is new and generalizes a recent result available in \cite{DMAS2023}.

\begin{theorem}
Let $(P_n)_{n\geq 0}$ be an OPS such that the following equation holds
\begin{align}\label{second-pb-to-solve}
&(az^2+bz+c)\mathcal{D}_q \mathcal{S}_q P_n(z)=a_n P_{n+1}(z)+b_n P_n(z)+c_n P_{n-1}(z)\;,
\end{align}
where $a,b,c\in \mathbb{C}$, and $(a_n)_{n\geq 0}$, $(b_n)_{n\geq 0}$, $(c_n)_{n\geq 0}$ are sequences of complex numbers such that $c_n\neq 0$ for each $n$. Then up to an affine transformation of the variable, $(P_n)_{n\geq 0}$ is the Askey-Wilson polynomial sequence or special or limiting case of it.
\end{theorem}

\begin{proof}
Let $(P_n)_{n\geq 0}$ be an OPS solution of \eqref{second-pb-to-solve}. First of all, we claim that 
\begin{align}\label{start-eq01}
\alpha \mathcal{S}_q ^2 P_n =\big(\alpha ^2 \texttt{U}_2 -\texttt{U}_1 ^2 \big)\mathcal{D}_q ^2P_n +\texttt{U}_1\mathcal{D}_q\mathcal{S}_qP_n +\alpha P_n\;,
\end{align} 
for each $n=0,1,\ldots$. Indeed this equation is obtained from \eqref{def-Sx^2f} by using \eqref{def-Sx-fg}, \eqref{def-DxnSxf} for $n=1$ and the fact that $$\mathcal{S}_q\texttt{U}_2=\alpha ^2\texttt{U}_2 +\texttt{U}_1 ^2,~\mathcal{D}_q\texttt{U}_2=2\alpha \texttt{U}_1,~\mathcal{D}_q\texttt{U}_1=\alpha ^2-1 ,~\mathcal{S}_q\texttt{U}_1=\alpha \texttt{U}_1  \;.$$
Now we apply the operator $S_q$ to the TTRR \eqref{TTRR_relation} satisfied by the monic OPS $(P_n)_{n\geq 0}$, solution of \eqref{second-pb-to-solve}. This gives after using \eqref{def-Sx-fg} the following relation
$$\texttt{U}_2 \mathcal{D}_qP_n +\alpha \mathcal{S}_qP_n =\mathcal{S}_qP_{n+1} +B_n\mathcal{S}_qP_n +C_n\mathcal{S}_qP_{n-1}  \;.$$
We also apply the operator $\mathcal{D}_q$ to the above equation using successively \eqref{def-Dx-fg}, \eqref{def-DxnSxf} (for $n=1$) and \eqref{start-eq01} to obtain
\begin{align}\label{almost-final01}
2\big(\alpha ^2\texttt{U}_2 -\texttt{U}_1 ^2 \big)\mathcal{D}_q ^2 P_n &+(4\alpha ^2-3)z\mathcal{D}_q\mathcal{S}_qP_n+\alpha P_n\nonumber\\
&=\mathcal{D}_q\mathcal{S}_qP_{n+1} +B_n\mathcal{D}_q\mathcal{S}_qP_n+C_n\mathcal{D}_q\mathcal{S}_qP_{n-1}\;.
\end{align}
We finally multiply \eqref{almost-final01} by the polynomial $az^2+bz+c$ and using successively the TTRR \eqref{TTRR_relation} and \eqref{second-pb-to-solve} to obtain
\begin{align}\label{final01-}
2\big(\alpha ^2\texttt{U}_2 -\texttt{U}_1 ^2 \big)&(az^2+bz+c)\mathcal{D}_q ^2 P_n\nonumber\\
&=r_n ^{[1]}P_{n+2} +r_n ^{[2]}P_{n+1}+r_n ^{[3]}P_n +r_n ^{[4]}P_{n-1}+r_n ^{[5]}P_{n-2}\;,
\end{align}
for each $n=0,1,\ldots$, where 
\begin{align*}
r_n ^{[1]}&= a_{n+1}-(4\alpha ^2 -3)a_n -\alpha a ,\\
r_n ^{[2]}&= b_{n+1}-(4\alpha ^2 -3)b_n -\alpha b +\big(B_n-(4\alpha ^2 -3)B_{n+1}\big)a_n -\alpha a(B_n+B_{n+1}),\\
r_n ^{[3]}&= c_{n+1}-(4\alpha ^2 -3)c_n -\alpha c -4(\alpha ^2-1)b_nB_n +a_{n-1}C_n -\alpha bB_n-v_n,\\
r_{n} ^{[4]}&=\big(B_n-(4\alpha ^2 -3)B_{n-1} \big)c_n -\big(\alpha b -b_{n-1}+(4\alpha ^2-3)b_n+\alpha a(B_n +B_{n-1})  \Big)C_n ,\\
r_n ^{[5]}&= c_{n-1}C_n-(4\alpha^2 -3)c_nC_{n-1} -\alpha aC_n C_{n-1},
\end{align*}
being $v_n=(4\alpha ^2-3)a_nC_{n+1} +\alpha a(C_{n+1}+B_n ^2+C_{n})$. In addition, it is not hard to see that $r_n ^{[5]}\neq 0 $ for each $n$. This equation is of the type \eqref{Ismail-conjecture} with $\pi(z)=2\big(\alpha ^2\texttt{U}_2 -\texttt{U}_1 ^2 \big)(az^2+bz+c)$ and hence the result follows by Theorem \ref{Kenfack-Kerstin-result}.

\end{proof}
Following the same ideas, one may prove the following one.
\begin{corollary}
The Askey-Wilson polynomials, or special or limiting case of them are the only OPS solutions of the following equation
\begin{align*}
&(az^2+bz+c)\mathcal{S}_q \mathcal{D}_q P_n(z)=a_n P_{n+1}(z)+b_n P_n(z)+c_n P_{n-1}(z)\;,
\end{align*}
where $a,b,c\in \mathbb{C}$ and $c_n\neq 0$ for each $n$.
\end{corollary}

\section{Special cases}\label{example}
In this section we consider some special cases of results presented in the previous part. For this purpose we use the following result.
\begin{theorem}\label{main-Thm1}\cite{KDP2021}
Let $(P_n)_{n\geq 0}$ be a monic OPS with respect to ${\bf u} \in \mathcal{P}^*$. 
Suppose that ${\bf u}$ satisfies the distributional equation
$${\bf D}_q(\phi {\bf u})={\bf S}_q(\psi {\bf u})\;,$$
where $\phi(z)=az^2+bz+c$ and $\psi(z)=dz+e$, with $d\neq0$.
Then $(P_n)_{n\geq 0}$ satisfies \eqref{TTRR_relation} with
\begin{align}
B_n  = \frac{\gamma_n e_{n-1}}{d_{2n-2}}
-\frac{\gamma_{n+1}e_n}{d_{2n}},\quad
C_{n+1}  =-\frac{\gamma_{n+1}d_{n-1}}{d_{2n-1}d_{2n+1}}\phi^{[n]}\left( -\frac{e_{n}}{d_{2n}}\right),\label{Bn-Cn-Dx}
\end{align}
 where $d_n=a\gamma_n+d\alpha_n$, $e_n=b\gamma_n+e\alpha_n$, and \begin{align*}
\phi^{[n]}(z)=\big(d(\alpha^2-1)\gamma_{2n}+a\alpha_{2n}\big)
\big(z^2-1/2\big)+\big(b\alpha_n+e(\alpha^2-1)\gamma_n\big)z+ c+a/2,
\end{align*}
\end{theorem}
Using this result, we obtain the following.
\begin{corollary}
The only monic OPS, $(P_n)_{n\geq 0}$, satisfying 
\begin{align}\label{case-zero-0}
\mathcal{D}_qP_{n+1}(z)=\alpha_n^{-1}\gamma_{n+1}\mathcal{S}_qP_n(z),
\end{align}
are those of the Askey-Wilson polynomials
$$
P_n(z)=Q_n\left( z;a,-a,iq^{-1/2}/a,-iq^{-1/2}/a\Big|q\right), 
$$
with $a\notin \left\lbrace \pm q^{(n-1)/2}, \pm iq^{-n/2}\,|\,n=0,1,\ldots\right\rbrace$.
\end{corollary}
\begin{proof}
This is easily proved by showing that all monic OPS solutions of \eqref{case-zero-0} are classical and therefore use the previous theorem to deduce the result. For more details we refer the reader to \cite{KCDMJP2021-a}. 
\end{proof}

\begin{corollary}\cite{DMAS2023}
The Rogers $q^{\pm 2}$-Hermite polynomials are the only OPS solutions of the following equation
\begin{align}\label{case-zero-01-q-Hermite}
\mathcal{D}_q\mathcal{S}_q P_{n}(z)=r_nP_{n-1}(z),\quad n=0,1,\ldots\;.
\end{align}
\end{corollary}

\begin{proof}
Let $(P_n)_{n\geq 0}$ be a monic OPS satisfying \eqref{case-zero-01-q-Hermite}. We claim that following system of difference equations holds
\begin{align}
&r_{n+2} ~-2(2\alpha ^2 -1)r_{n+1} +r_n  =0\;,\label{eqS11}\\
&t_{n+2} -2(2\alpha ^2-1)t_{n+1}+t_n=0,~\quad t_n:=r_n/C_n\;, \label{eqS22}\\
&r_{n+1}B_{n+1}-(4\alpha ^2 -3)(r_n+r_{n+1})B_n +r_nB_{n-1}=0\;,\label{eqS33}\\
&t_{n+3}B_{n+2}-(t_{n+2}+t_{n+1})B_{n+1}+t_{n}B_{n}=0\;,\label{eqS44}\\
t_{n+2}&(C_{n+1}-1/4)-2t_n(C_n-1/4)+t_{n-2}(C_{n-1}-1/4)\nonumber \\
&\quad \quad \quad \quad \quad \quad =t_n\left[B_n ^2 -2(2\alpha ^2 -1)B_nB_{n-1}+B_{n-1} ^2  \right]\;,\label{eqS55}
\end{align}
where $B_n$ and $C_n$ are the coefficients of the TTRR \eqref{TTRR_relation} satisfied by $(P_n)_{n\geq 0}$. Indeed, applying the operator $\mathcal{D}_q ^2$ to the TTRR \eqref{TTRR_relation} taking into account \eqref{def-Dx-fg} and \eqref{def-DxnSxf} gives
\begin{align}\label{equatwithDx^2Pn}
2\mathcal{D}_q\mathcal{S}_qP_n(z) +(z-B_n)\mathcal{D}_q ^2P_nz(z)=\mathcal{D}_q ^2P_{n+1}(z)+ C_n\mathcal{D}_q ^2P_{n-1}(z)\;.
\end{align}
Multiplying this equation by $2\big(\alpha ^2\texttt{U}_2-\texttt{U}_1 ^2\big)$, we successively use \eqref{final01-} with $a=0=b$ and $c=1$ therein together \eqref{case-zero-01-q-Hermite} and the TTRR \eqref{TTRR_relation} to obtain a vanishing linear combination of polynomials $P_{n+1}$, $P_n$, $P_{n-2}$, $P_{n-2}$ and $P_{n-3}$. Therefore the coefficients of the mentioned linear combination must all be zero. This gives the system \eqref{eqS11}--\eqref{eqS55}. Now by identifying the two first coefficients of terms with higher degree in \eqref{case-zero-01-q-Hermite}, on may obtain
$$r_n=\gamma_{2n}/2,\quad B_n=B_0\alpha_{2n+1}/\alpha \;\quad n=0,1,\ldots\;.$$ 
On the other hands solutions of \eqref{eqS22} are given by $$t_n=k_1q^n +k_2q^{-n}\;,\quad k_1,k_2\in \mathbb{C}, \quad |k_1|+|k_2|\neq 0\;.$$ From this, it is not hard to see that the above expression of $B_n$ satisfies \eqref{eqS33} if and only if $B_0$ and consequently $$B_n=0,\quad n=0,1,\ldots\;.$$ Finally from the expression of $C_n$ obtained in \eqref{eqS22}, \eqref{eqS55} is satisfied if and only if $k_1k_2=0$ and therefore we obtain $$C_{n+1}=(1-q^{2n+2})/4\quad \textit{or}\quad C_{n+1}=(1-q^{-2n-2})/4,\quad n=0,1,\ldots \;.$$ The results follows.
\end{proof}

\begin{remark}
The results obtained here were proved for the $q$-quadratic lattice
$x(s)=(q^{-s} +q^{s})/2$ and can be extended to the quadratic lattice $x(s)=\mathfrak{c}_4s^2+\mathfrak{c}_5s+\mathfrak{c}_6$ by taking the appropriate limit as it was discussed in \cite{KDP2021}. 
\end{remark}

\section*{Acknowledgements }
The author D. Mbouna was partially supported by CMUP, member of LASI, which is financed by national funds through FCT - Fundac\~ao para a Ci\^encia e a Tecnologia, I.P., under the projects with reference UIDB/00144/2020 and UIDP/00144/2020. A. Suzuki is supported by the FCT grant 2021.05089.BD and partially supported by the Centre for Mathematics of the University of Coimbra-UIDB/00324/2020, funded by the Portuguese Government through FCT/ MCTES.

{

\end{document}